\documentclass{interact}

\usepackage[caption=false]{subfig}

\usepackage[numbers,sort&compress]{natbib}
\bibpunct[, ]{[}{]}{,}{n}{,}{,}
\makeatletter
\def\NAT@def@citea{\def\@citea{\NAT@separator}}
\makeatother

\usepackage{amsmath, amsthm, amssymb}
\usepackage{thmtools}
\usepackage{mathtools}
\mathtoolsset{centercolon}
\usepackage{enumitem}
\usepackage{graphicx}
\usepackage{delimset}
\usepackage[hidelinks]{hyperref}
\hypersetup{
  pdftitle={Closedness of the orbit-closed C-numerical range and submajorization},
  pdfauthor={Jireh Loreaux, Sasmita Patnaik}
}
\usepackage{tikz}
\usetikzlibrary{cd}
\usepackage{cleveref}

\newcommand{\set}{\delim\{\}}
\newcommand{\setb}{\delimpair\{{[m]\vert}\}}
\newcommand{\innerprod}{\delimpair<{[.],}>}
\newcommand{\seq}{\delim()}

\expandafter\mathchardef\expandafter\varphi\number\expandafter\phi\expandafter\relax
\expandafter\mathchardef\expandafter\phi\number\varphi

\makeatletter
\newcommand\bigcdot{\mathpalette\bigcdot@{.5}}
\newcommand*\bigcdot@[2]{\mathbin{\vcenter{\hbox{\scalebox{#2}{$\m@th#1\bullet$}}}}}
\makeatother

\newcommand{\term}[1]{\emph{#1}}
\newcommand{\weakstar}{weak$^{*}$}
\newcommand{\wstar}{w^{*}}

\newcommand{\Hil}{\mathcal{H}}
\newcommand{\K}{\mathcal{K}}

\newcommand{\traceclass}{\mathcal{L}_1}
\newcommand{\orbit}{\mathcal{O}}
\newcommand{\ugroup}{\mathcal{U}}

\newcommand{\zop}{\mathbf{0}}

\renewcommand{\epsilon}{\varepsilon}

\newcommand{\nats}{\mathbb{N}}

\newcommand{\reals}{\mathbb{R}}
\newcommand{\complex}{\mathbb{C}}

\DeclareFontFamily{U}{mathb}{\hyphenchar\font45}
\DeclareFontShape{U}{mathb}{m}{n}{
<-6> mathb5 <6-7> mathb6 <7-8> mathb7
<8-9> mathb8 <9-10> mathb9
<10-12> mathb10 <12-> mathb12
}{}
\DeclareSymbolFont{mathb}{U}{mathb}{m}{n}
\DeclareMathSymbol{\pprec}{\mathrel}{mathb}{"CE}
\DeclareMathSymbol{\ssucc}{\mathrel}{mathb}{"CF}

\DeclareMathOperator{\spec}{\sigma}
\newcommand{\essspec}{\spec_{\mathrm{ess}}}

\DeclareMathOperator{\diag}{diag}
\DeclareMathOperator{\rank}{rank}
\DeclareMathOperator{\trace}{Tr}
\DeclareMathOperator{\spans}{span}
\DeclareMathOperator{\conv}{conv}
\DeclareMathOperator{\ext}{ext}

\newcommand{\closure}[2][]{\overline{#2}^{#1}}

\newcommand{\nr}{W}
\newcommand{\essnr}{\nr_{\textrm{ess}}}
\newcommand{\knr}[1][k]{\nr_{#1}}
\newcommand{\cnr}[1][C]{\nr_{#1}}
\newcommand{\ocnr}[1][C]{\nr_{\orbit(#1)}}
\newcommand{\maj}{\prec}
\newcommand{\submaj}{\pprec}

\setlist[enumerate]{label=\textup{(\roman*)}}

\numberwithin{equation}{section}

\theoremstyle{plain} 
\newtheorem{theorem}{Theorem}[section]
\newtheorem{corollary}[theorem]{Corollary}
\newtheorem{lemma}[theorem]{Lemma}
\newtheorem{proposition}[theorem]{Proposition}

\theoremstyle{definition}
\newtheorem{definition}[theorem]{Definition}

\theoremstyle{remark}
\newtheorem{remark}[theorem]{Remark}
\newtheorem{example}[theorem]{Example}

\newcounter{case}
\makeatletter\@addtoreset{case}{theorem}\makeatother

\title{Closedness of the orbit-closed \\ $C$-numerical range and submajorization}

\author{\name{Jireh Loreaux\textsuperscript{a}\thanks{J. Loreaux email: jloreau@siue.edu} and Sasmita Patnaik\textsuperscript{b}\thanks{S. Patnaik email: sasmita@iitk.ac.in}}\affil{\textsuperscript{a}Southern Illinois University Edwardsville, 1 Hairpin Dr, Edwardsville, IL, 62026, USA; \textsuperscript{b}Indian Institute of Technology, Kanpur, Kalyanpur, Kanpur-208016, India.}}

\begin{document}

\maketitle

\begin{abstract}
  For a positive trace-class operator $C$ and a bounded operator $A$, we provide an explicit description of the closure of the orbit-closed $C$-numerical range of $A$ in terms of those operators submajorized by $C$ and the essential numerical range of $A$.
  This generalizes and subsumes recent work of Chan, Li and Poon for the $k$-numerical range, as well as some of our own previous work on the orbit-closed $C$-numerical range.
\end{abstract}

\begin{keywords}
  numerical range, $C$-numerical range, convex, trace-class, essential numerical range, majorization, submajorization, \weakstar{} convergence
\end{keywords}
\begin{amscode}
  Primary 47A12, 47B15; Secondary 52A10, 52A40, 26D15.
\end{amscode}

\section{Introduction}

Herein we let $\Hil$ denote a separable complex Hilbert space and $B(\Hil)$ the collection of all bounded linear operators on $\Hil$.
For $A \in B(\Hil)$, the \term{numerical range} $\nr(A)$ is the image of the unit sphere of $\Hil$ under the continuous quadratic form $x \mapsto \innerprod{Ax}{x}$, where $\innerprod{\bigcdot}{\bigcdot}$ denotes the inner product on $\Hil$.
The \term{essential numerical range} $\essnr(A)$ has many equivalent definitions, including the set of limits of convergent sequences $\innerprod{Ax_n}{x_n}$ where $\seq{x_n}_{n=1}^{\infty}$ is an orthonormal sequence \cite{FSW-1972-ASM}.
It is well-known that
\begin{equation}
  \label{eq:closure-numerical-range}
  \closure{W(A)} = \conv (\nr(A) \cup \essnr(A)),
\end{equation}
which is due to Lancaster \cite{Lan-1975-PAMS}.

There have been a number of generalizations of this result, including by Chan \cite{Cha-2018-LaMA} and Chan, Li and Poon \cite{CLP-2020-LaMA}.
Chan \cite{Cha-2018-LaMA} generalized this to the joint numerical range of an $n$-tuple of operators, whereas Chan, Li and Poon \cite{CLP-2020-LaMA} generalized it to the $k$-numerical range.
The \term{$k$-numerical range} is the collection\footnotemark{}
\begin{equation}
  \label{eq:k-numerical-range-definition}
  \knr(A) := \setb{\trace(PA)}{P\ \text{rank-$k$ projection}}.
\end{equation}
The reader should note that this is a natural generalization of the standard numerical range since $\knr[1](A) = \nr(A)$.
In \cite{CLP-2020-LaMA}, Chan, Li and Poon showed
\begin{equation}
  \label{eq:clp-closure-k-numerical-range}
  \closure{\knr(A)} = \conv \bigcup_{j=0}^k ( \knr[j](A) + (k-j) \essnr(A) ),
\end{equation}
which generalizes \eqref{eq:closure-numerical-range}.
They also proved if $\knr[k+1](A)$ is closed, then $\knr(A)$ is closed.
Consequently, $\knr(A)$ is closed if and only if
\begin{equation}
  \label{eq:k-numerical-range-inclusion-chain}
  k \essnr(A) \subseteq \knr[1](A) + (k-1) \essnr(A) \subseteq \cdots \subseteq \knr[k-1](A) + \essnr(A) \subseteq \knr(A).
\end{equation}
Chan, Li and Poon very recently (\cite{CLP-2021}) extended some of their results to the \term{joint $k$-numerical range} of a tuple of operators.

\footnotetext{%
  This definition of the $k$-numerical range is the one given by Chan, Li and Poon.
  However, the reader should be aware that there is another definition, differing only by a scaling factor:
  \begin{equation*}
    \setb*{\frac{1}{k} \trace(PA)}{P\ \text{rank-$k$ projection}}.
  \end{equation*}
  The definition given in this footnote is the one originally described by Halmos in \cite{Hal-1964-ASM}.
  Both definitions appear throughout the literature, so one always has to be careful to see which definition the authors use.
}

In our recent paper \cite{LP-2021-LaMA}, we introduced the \term{orbit-closed $C$-numerical range} for $C \in \traceclass$, the trace class, which is defined as follows.
Let $\ugroup(C)$ denote the \term{unitary orbit} of $C$ under the natural conjugation action of the unitary group.
Then the trace-norm closure $\orbit(C) := \closure[\norm{\bigcdot}_1]{\ugroup(C)}$ we call the \term{orbit} of $C$.
The orbit-closed $C$-numerical range is
\begin{equation}
  \label{eq:orbit-closed-c-numerical-range}
  \ocnr(A) := \setb{\trace(XA)}{X \in \orbit(C)}.
\end{equation}
For finite rank $C$, $\orbit(C) = \ugroup(C)$ and hence, in this case, $\ocnr(A)$ coincides with the usual $C$-numerical range $\cnr(A)$ initially studied by Westwick \cite{Wes-1975-LMA}, Goldberg and Straus \cite{GS-1977-LAA}, and by many others since then.
Consequently, $\ocnr(A)$ constitutes a natural extension of this object to the setting when $C$ has infinite rank.

This paper generalizes the aforementioned results (\eqref{eq:clp-closure-k-numerical-range} and \eqref{eq:k-numerical-range-inclusion-chain}) of Chan, Li and Poon \cite{CLP-2020-LaMA} to the context of the orbit-closed $C$-numerical range.
Our main theorem (\Cref{thm:main-theorem}) generalizes \eqref{eq:clp-closure-k-numerical-range} and provides an independent proof of this fact.
Moreover, as we showed in \cite{LP-2021-LaMA} that $\ocnr(A)$ is intimately connected with majorization\footnotemark{} (denoted $\maj$, see \Cref{def:majorization}) when $C$ is selfadjoint, so also we connect the closure $\closure{\ocnr(A)}$ to submajorization (denoted $\submaj$, see \Cref{def:majorization}).
In particular, we prove in \Cref{thm:main-theorem}, for $C \in \traceclass^+$ and $A \in B(\Hil)$,
\footnotetext{%
  Poon also made this connection in the case when $C$ is finite rank \cite{Poo-1980-LMA}.%
}
\begin{align*}
  \closure{\ocnr(A)} &= \setb1{ \trace(XA) + \trace(C-X) \essnr(A) }{ X \in \traceclass^+, \lambda(X) \submaj \lambda(C) } \\
                     &= \conv \bigcup_{0 \le m \le \rank(C)} \big( \ocnr[C_m](A) + \trace(C-C_m) \essnr(A) \big),
\end{align*}
where $C_m := \diag(\lambda_1(C),\ldots,\lambda_m(C),0,0,\ldots)$.
\Cref{ex:selfadjoint-fail} shows that \Cref{thm:main-theorem} doesn't generalize to $C$ selfadjoint in the way one might expect.

We later establish in \Cref{thm:closedness-restricts} that, when $C$ is positive, if $\ocnr(A)$ is closed, then $\ocnr[C_m](A)$ is closed for every $0 \le m < \rank(C)$.
Combining \Cref{thm:main-theorem,thm:closedness-restricts} yields \Cref{cor:closed-inclusion-chain}: $\ocnr(A)$ is closed if and only if
\begin{align*}
  \trace(C) \essnr(A) &\subseteq \ocnr[C_1](A) + \trace(C-C_1) \essnr(A) \\
                      &\subseteq \ocnr[C_2](A) + \trace(C-C_2) \essnr(A) \\
                      &\ \,\vdots \\
                      &\subseteq \ocnr(A),
\end{align*}
which generalizes \eqref{eq:k-numerical-range-inclusion-chain}.
Note: if $\rank(C)$ is infinite, this chain of inclusions has order type $\omega + 1$.

\section{Notation and Background}

We first introduce relevant notation.
We let $\K$ denote the norm-closed ideal of $B(\Hil)$ consisting of compact operators, and we let $\traceclass$ denote the ideal of trace-class operators, which is a Banach space when equipped with the \term{trace norm} $\norm{C}_1 := \trace(\abs{C})$.
The collection of selfadjoint elements in these classes are denoted $\K^{sa}, B(\Hil)^{sa}, \traceclass^{sa}$, respectively.
The positive elements are likewise denoted $\K^+, B(\Hil)^+, \traceclass^+$.
For any $X \in B(\Hil)$, $R_X$ denotes the range projection of $X$.
The symbol $\zop_{\Hil}$ denotes the zero operator on the (separable infinite-dimensional) Hilbert space $\Hil$.

For $A \in B(\Hil)$, the real and imaginary parts of $A$ are given by $\Re(A), \Im(A)$.
For a selfadjoint operator $A \in B(\Hil)^{sa}$, we let $A_{\pm}$ denote its positive and negative parts.
If $A$ is selfadjoint and $E \subseteq \reals$ is Borel, $\chi_E(A)$ is the spectral projection of $A$ from the Borel functional calculus corresponding to the set $E$.

If $C \in \K$, we let $\lambda(C)$ represent the\footnotemark{} \term{eigenvalue sequence} of $C$, which consists of the eigenvalues of $C$ listed in order of nonincreasing modulus, repeated according to algebraic multiplicity, and omitting the zero eigenvalue if there are infinitely many nonzero eigenvalues.

\footnotetext{%
  Note that $\lambda(C)$ is not generally uniquely determined since there may be unequal eigenvalues with the same modulus.
  However, if $C$ is positive, then $\lambda(C)$ is uniquely determined.%
}

If $C \in \K^{sa}$, then $\lambda^+(C)$ is the nonincreasing rearrangement of the sequence of nonnegative eigenvalues, along with infinitely many zeros when $\rank(C_+) < \infty$ (even if zero is \emph{not} an eigenvalue of $C$).
Similarly, $\lambda^-(C) := \lambda^+(-C)$.
Note that if $C \in \K^+$, then $\lambda(C) = \lambda^+(C)$.

We let $\ugroup(C)$ denote the \term{unitary orbit} of $C$, and, for $C \in \traceclass$,  $\orbit(C) := \closure[\norm{\cdot}_1]{\ugroup(C)}$.
We note that for normal $C \in \traceclass$, the following are equivalent (see \cite[Proposition~3.1]{LP-2021-LaMA}): $X \in \orbit(C)$; $X$ is normal and $\lambda(X) = \lambda(C)$; $X \oplus \zop_{\Hil} \in \ugroup(C \oplus \zop_{\Hil})$.

Throughout this paper, whenever we refer to the \weakstar{} topology, we will always mean the topology on $\traceclass \cong \K^{*}$ induced by the isometric isomorphism $C \mapsto \trace(C \bigcdot)$.
Since $\K$ is separable, this topology is metrizable on trace-norm bounded sets (by Banach--Alaoglu) and \weakstar{} convergence $X_n \to X$ means $\trace(X_n A) \to \trace(XA)$ for all $A \in \K$ (careful, \emph{not} $A \in B(\Hil)$).

\begin{definition}
  \label{def:majorization}
  Suppose that $a := (a_k)_{k=1}^{\infty}, b := (b_k)_{k=1}^{\infty}$ are real-valued sequences converging to zero.
  If for all $n \in \nats$,
  \begin{equation*}
    \sum_{k=1}^n a^{\pm}_k \le \sum_{k=1}^n b^{\pm}_k,
  \end{equation*}
  then $a$ is \term{submajorized} by $b$, denoted $a \submaj b$.

  If $a,b \in \ell_1$, and $a \submaj b$ and also
  \begin{equation*}
    \sum_{k=1}^{\infty} a_k = \sum_{k=1}^{\infty} b_k,
  \end{equation*}
  then $a$ is \term{majorized} by $b$, denoted $a \maj b$.
\end{definition}

This concludes the necessary notation.
We now review some background material which will be necessary, most of which comes from \cite{LP-2021-LaMA}.
Our main result from \cite{LP-2021-LaMA} is:

\begin{theorem}[\protect{\cite[Theorem~4.1, Corollaries~4.1,~4.2]{LP-2021-LaMA}}]
  \label{thm:c-numerical-range-via-majorization}
  For a selfadjoint trace-class operator $C \in \traceclass^{sa}$ and any $A \in B(\Hil)$,
  \begin{equation*}
    \ocnr(A) = \{ \trace(XA) \mid X \in \traceclass^{sa}, \lambda(X) \maj \lambda(C) \}.
  \end{equation*}
  Consequently,
  \begin{enumerate}
  \item $\ocnr(A)$ is convex.
  \item \label{item:majorization-inclusion} If $C' \in \traceclass^{sa}$ and $\lambda(C) \maj \lambda(C')$, then $\ocnr(A) \subseteq \ocnr[C'](A)$.
  \end{enumerate}
\end{theorem}

We will need a simple lemma which in some sense allows us to stay inside a given subspace in \Cref{thm:c-numerical-range-via-majorization}\ref{item:majorization-inclusion}.

\begin{lemma}
  \label{lem:subspace-restriction-majorization}
  Let $C,C' \in \traceclass^{sa}$ be selfadjoint trace-class operators with $\lambda(C) \maj \lambda(C')$, and let $A \in B(\Hil)$, $X \in \orbit(C)$.

  If $\rank(C) \ge \rank(C')$, then there is some $X' \in \orbit(C')$ for which $R_{X'} \le R_X$ and $\trace(XA) = \trace(X'A)$.
  Moreover, if $C' \ge 0$, the hypothesis $\rank(C) \ge \rank(C')$ may be omitted.
\end{lemma}

\begin{proof}
  We first note that, in the context $C' \ge 0$, the hypothesis $\lambda(C) \maj \lambda(C')$ implies $C \ge 0$ and $\rank(C) \ge \rank(C')$, which is just a simple fact about majorization of nonnegative sequences.
  Indeed, if $\rank(C) = \infty$, there is nothing to prove, so we may assume $\rank(C) < \infty$.
  Then since $\lambda(C) \maj \lambda(C')$ (and so $\lambda(C) \submaj \lambda(C')$),
  \begin{equation*}
    \trace(C) = \sum_{n=1}^{\rank(C)} \lambda_n(C) \le \sum_{n=1}^{\rank(C)} \lambda_n(C') \le \sum_{n=1}^{\infty} \lambda_n(C') = \trace(C').
  \end{equation*}
  Because $\lambda(C) \maj \lambda(C')$, then $\trace(C) = \trace(C')$ and so we must have equality throughout this chain.
  Therefore $\sum_{n=\rank(C) + 1}^{\infty} \lambda_n(C') = 0$, and thus $\rank(C') \le \rank(C)$.

  Now suppose $C,C'$ are selfadjoint and $\rank(C) \ge \rank(C')$.
  Since $X$ is selfadjoint, $R_X X = X R_X = X$.
  Therefore
  \begin{equation*}
    \trace(XA) = \trace(R_X X R_X A) = \trace(R_X X R_X A R_X) = \trace_{R_X \Hil}(YA'),
  \end{equation*}
  where $Y = R_X X \vert_{R_X \Hil} \in B(R_X \Hil)$ and $A' = R_X A \vert_{R_X \Hil} \in B(R_X \Hil)$.
  Since $R_X X R_X = X$, then $\lambda(Y) = \lambda(X) = \lambda(C)$.
  Because $\rank(C') \le \rank(C) = \dim R_X \Hil$, there is some selfadjoint $C''$ acting on $R_X \Hil$ with $\lambda(C'') = \lambda(C')$.
  Therefore, $\lambda(Y) \maj \lambda(C'')$ and so by \Cref{thm:c-numerical-range-via-majorization}(ii), $\cnr[\orbit_{R_X \Hil}(Y)](A') \subseteq \cnr[\orbit_{R_X \Hil}(C'')](A')$.
  Consequently, there is some $Y' \in \orbit_{R_X \Hil}(C'')$ for which $\trace_{R_X \Hil}(YA') = \trace_{R_X \Hil}(Y'A')$.
  Finally, set $X' := Y' \oplus \zop_{R_X^{\perp} \Hil}$, so that $R_{X'} \le R_X$, and $\lambda(X') = \lambda(Y') = \lambda(C'') = \lambda(C')$, and hence $X' \in \orbit(C')$, and also
  \begin{equation*}
    \trace(X'A) = \trace_{R_X \Hil}(Y'A') = \trace_{R_X \Hil}(YA') = \trace(XA). \qedhere
  \end{equation*}
\end{proof}

\begin{remark}
  In case $C'$ is selfadjoint, the hypothesis $\rank(C) \ge \rank(C')$ may not be omitted in general.
  Indeed, there are examples of selfadjoint $C,C'$ such that $\lambda(C) \maj \lambda(C')$, but for which $\rank(C) < \rank(C')$, thereby ensuring the conclusion of \Cref{lem:subspace-restriction-majorization} is unattainable.
  For example, if $C'$ is selfadjoint and trace zero, then it majorizes the zero operator.
\end{remark}

\Cref{lem:subspace-restriction-majorization} has the following corollary in the extremal case when $C'$ is rank-$2$, or, in case $C$ is positive or negative, even rank-$1$.

\begin{corollary}
  \label{cor:subspace-restrction-rank-2-majorization}
  Suppose that $C \in \traceclass^{sa}$ is a selfadjoint trace-class operator and $X \in \orbit(C)$.
  Then if $C':= \diag(\trace(C_+), -\trace(C_-), 0, \ldots)$, there is some $X' \in \orbit(C')$ with $R_{X'} \le R_X$ and $\trace(X'A) = \trace(XA)$.
\end{corollary}

\begin{proof}
  Notice $\lambda(C) \maj \lambda(C')$ trivially.
  If either $C \ge 0$ or $C \le 0$, then the result follows immediately from \Cref{lem:subspace-restriction-majorization}.
  Also, if $C$ is selfadjoint and is neither positive or negative, then $\rank(C) \ge 2 = \rank(C')$, so the result again follows from \Cref{lem:subspace-restriction-majorization}.
\end{proof}

Two more results which will be vital for us in this paper concern the supremum of the orbit-closed $C$-numerical range of selfadjoint operators.

\begin{proposition}[\protect{\cite[Proposition~5.1]{LP-2021-LaMA}}]
  \label{prop:c-numerical-range-maximum}
  Let $C \in \traceclass^+$ be a positive trace-class operator and let $A \in \K^+$ be a positive compact operator.
  Then
  \begin{equation*}
    \sup \ocnr(A) = \sum_{n=1}^{\infty} \lambda_n(C) \lambda_n(A),
  \end{equation*}
  and moreover the supremum is attained.
\end{proposition}

We note that in the theorem below, $(A-mI)_+$ is a positive compact operator, so it is subject to \Cref{prop:c-numerical-range-maximum}.

\begin{theorem}[\protect{\cite[Theorem~5.2]{LP-2021-LaMA}}]
  \label{thm:c-numerical-range-selfadjoint-formula}
  Let $C \in \traceclass^+$ be a positive trace-class operator and suppose $A \in B(\Hil)$ is selfadjoint.
  Let $m := \max \essspec(A)$.
  Then
  \begin{equation*}
    \sup \ocnr(A) = m\trace C + \sup \ocnr(A-mI)_+,
  \end{equation*}
  Moreover, letting $P := \chi_{[m,\infty)}(A)$, then $\sup \ocnr(A)$ is attained if and only if $\rank(C) \le \trace(P)$.
  In fact, when $X \in \orbit(C)$ attains the supremum, $XP = PX = X$.
\end{theorem}

Hiai and Nakamura established in \cite{HN-1991-TAMS} the following connection between submajorization of eigenvalue sequences of selfadjoint operators and closed convex hulls of unitary orbits.

\begin{proposition}[\protect{\cite[Theorem~3.3]{HN-1991-TAMS}}]
  \label{prop:wot-closure-convex-orbit-weak-majorization}
  For a selfadjoint compact operator $C \in \K^{sa}$,
  \begin{equation*}
    \{ X \in \K^{sa} \mid \lambda(X) \submaj \lambda(C) \} = \closure[\mathrm{wot}]{\conv\ugroup(C)} = \closure[\norm{\bigcdot}]{\conv\ugroup(C)}.
  \end{equation*}
\end{proposition}

Note that for $C$ trace-class, since the trace-norm topology on $\conv \ugroup(C)$ is stronger than the norm topology (or the weak operator topology), we may replace $\ugroup(C)$ in \Cref{prop:wot-closure-convex-orbit-weak-majorization} with $\orbit(C)$.
\Cref{prop:wot-closure-convex-orbit-weak-majorization} also has consequences for the \weakstar{} closure of the convex hull of the unitary orbit of a selfadjoint operator.

\begin{corollary}
  \label{cor:submajorization-weak-star-compact}
  For a selfadjoint trace-class operator $C \in \traceclass^{sa}$, 
  \begin{equation*}
    \setb{ X \in \traceclass^{sa} }{ \lambda(X) \submaj \lambda(C) } = \closure[\wstar]{\conv\ugroup(C)},
  \end{equation*}
  and moreover this set is compact and metrizable in the \weakstar{} topology on $\traceclass$, hence also \weakstar{} sequentially compact.
\end{corollary}

\begin{proof}
  We first remark that if $C \in \traceclass^{sa}$ and $X \in \K^{sa}$ with $\lambda(X) \submaj \lambda(C)$, then $X \in \traceclass$.
  Therefore, by \Cref{prop:wot-closure-convex-orbit-weak-majorization}
  \begin{align*}
    \setb{ X \in \traceclass^{sa} }{ \lambda(X) \submaj \lambda(C) }
    &= \setb{ X \in \K^{sa} }{ \lambda(X) \submaj \lambda(C) } \\
    &= \closure[\mathrm{wot}]{\conv\ugroup(C)} \\
    &= \closure[\norm{\bigcdot}]{\conv\ugroup(C)}.
  \end{align*}
  
  Since the \weakstar{} topology on $\traceclass$ is weaker, on trace-norm bounded sets, than the (operator) norm topology and stronger than the weak operator topology, we conclude that trace-norm bounded subsets of $\traceclass$ which are both weak operator closed and (operator) norm closed are also \weakstar{} closed.

  Finally, we note that $\setb{ X \in \traceclass^{sa} }{ \lambda(X) \submaj \lambda(C) }$ is trace-norm bounded by $\norm{C}_1$, and so the previous paragraph guarantees
  \begin{equation*}
    \setb{ X \in \traceclass^{sa} }{ \lambda(X) \submaj \lambda(C) } = \closure[\wstar]{\conv\ugroup(C)}.
  \end{equation*}
  Finally, the Banach--Alaoglu theorem implies that this set, being trace-norm bounded and \weakstar{} closed, is \weakstar{} compact and the \weakstar{} topology is metrizable (on this trace-norm bounded set), the latter because $\traceclass \cong \K^{*}$ and $\K$ is separable.
  Because compactness and sequential compactness are equivalent in metric spaces, this set is \weakstar{} sequentially compact as well.
\end{proof}

A cursory examination of the proof of \cite[Lemma~5.1]{LP-2021-LaMA} affords us the following result relating to extreme points (see \Cref{cor:extreme-points-submjaorization} for the connection) of the collection of operators whose eigenvalue sequences are submajorized by that of a fixed trace-class operator.

\begin{lemma}[\protect{\cite[Proof of Lemma~5.1]{LP-2021-LaMA}}]
  \label{lem:submajorization-convex-hull-of-orbits}
  For selfadjoint trace-class operators $X,C \in \traceclass^{sa}$ with $\lambda(X) \submaj \lambda(C)$, there is some $Y \in \traceclass^{sa}$ for which
  \begin{enumerate}
  \item $\lambda(X) \maj \lambda(Y) \submaj \lambda(C)$;
  \item $\orbit(Y) \subseteq \conv \bigcup \setb{ \orbit(C_{m_-,m_+}) }{ 0 \le m_{\pm} \le \rank(C_{\pm})}$,
  \end{enumerate}
  where $C_{m_-,m_+}$ is the operator $C (P^-_{m_-} + P^+_{m_+})$ where $\trace(P^{\pm}_{m_{\pm}}) = m_{\pm}$, and for some $\lambda_- \le 0 \le \lambda_+$, $\chi_{(-\infty,\lambda_-)}(C) \le P^-_{m_-} \le \chi_{(-\infty,\lambda_-]}(C)$ and $\chi_{(\lambda_+,\infty)}(C) \le P^+_{m_+} \le \chi_{[\lambda_+,\infty)}(C)$.

  In other words, $C_{m_-,m_+}$ is the selfadjoint operator whose eigenvalues are the smallest $m_-$ negative eigenvalues $C$ along with the largest $m_+$ positive eigenvalues of $C$, namely $-\lambda^-_1(C),\ldots,-\lambda^-_{m_-}(C)$ and $\lambda^+_1,(C),\ldots,\lambda^+_{m_+}(C)$, along with the eigenvalue $0$ repeated with multiplicity $\trace(I-P_{m_-}^- - P_{m_+}^+)$.
\end{lemma}

Actually, it is possible to prove the following fact as well, but we will not use it; it's slightly too weak for the purposes of this paper.
For this reason, we omit the proof but note that it can be obtained from \Cref{lem:submajorization-convex-hull-of-orbits} and \cite[Lemma~4.1]{LP-2021-LaMA}.

\begin{corollary}
  \label{cor:extreme-points-submjaorization}
  If $C \in \traceclass^{sa}$, then
  \begin{equation*}
    \ext \setb{ X \in \traceclass^{sa} }{ \lambda(X) \submaj \lambda(C) } \subseteq \bigcup \setb{ \orbit(C_{m_-,m_+})} { 0 \le m_{\pm} \le \rank(C_{\pm}) }.
  \end{equation*}
  Moreover, if $C \ge 0$, then the above inclusion is an equality.
\end{corollary}

\section{Submajorization and the closure}
\label{sec:submajorization-closure}

In this section we prove our main theorem which characterizes $\closure{\ocnr(A)}$ in terms of submajorization and the essential numerical range (see \Cref{thm:main-theorem}).
We begin with a basic result concerning trace-class operators which converge \weakstar{} to zero.

\begin{lemma}
  \label{lem:weak-star-tail-convergence}
  Suppose that $\seq{Y_k}_{k=1}^{\infty}$ is a sequence in $\traceclass$ with $Y_k \xrightarrow{\wstar} 0$.
  For any finite projection $P$ and $A \in B(\Hil)$,
  \begin{equation*}
     \trace(Y_k A) - \trace(P^{\perp} Y_k P^{\perp} A) \to 0.
  \end{equation*}
\end{lemma}

\begin{proof}
  Notice that $Y_k A - P^{\perp} Y_k P^{\perp} A = P Y_k A + P^{\perp} Y_k P A$.
  Then
  \begin{align*}
    \trace(Y_k A) - \trace(P^{\perp} Y_k P^{\perp} A)
    &= \trace(P Y_k A) + \trace(P^{\perp} Y_k P A) \\
    &= \trace(Y_k (AP)) + \trace( Y_k (PAP^{\perp}))
  \end{align*}
  converges to zero since $AP,PAP^{\perp} \in \K$ and $Y_k \xrightarrow{\wstar} 0$.
\end{proof}

We now prove the key technical lemma.

\begin{lemma}
  \label{lem:key-lemma}
  Let $\seq{Y_k}_{k=1}^{\infty} \subseteq \traceclass^{sa}$ be a sequence of selfadjoint trace-class operators such that $\trace(Y_k) = c$ is a positive constant and $Y_k \xrightarrow{\wstar} 0$.
  Moreover, suppose that there is some $X \in \traceclass^+$ for which $Y_k \ge -X$ for all $k \in \nats$.
  If $A \in B(\Hil)$ and $\trace(Y_k A) \to \mu$, then for any finite projection $P$ and any $\epsilon > 0$ there is some $y \in P^{\perp} \Hil$ such that $\abs{c \innerprod{Ay}{y} - \mu} < \epsilon$.
\end{lemma}

\begin{proof}
  Suppose that $\seq{Y_k}_{k=1}^{\infty}$, $A$, $\mu$ are given with the properties and relationships specified in the statement.
  Let $\epsilon > 0$ and suppose that $P$ is any finite projection.

  Since $\trace(Y_k A) \to \mu$, there is some $N_1$ such that for all $k \ge N_1$
  \begin{equation}
    \label{eq:tr-Y_k-A-close-to-mu}
    \abs{\trace(Y_k A) - \mu} < \gamma := \frac{\epsilon}{3}.
  \end{equation}
  Let $Q$ be a finite spectral projection of $X$ so that $\trace(Q^{\perp} X Q^{\perp}) < \delta := \frac{\epsilon}{3(\norm{A} + 1)}$.
  Let $R = P \vee Q$ be the projection onto $P\Hil + Q\Hil$, which is a finite projection since both $P,Q$ are finite.
  Note that for all $k$, since $Y_k \ge -X$, then $R^{\perp} Y_k R^{\perp} \ge - R^{\perp} X R^{\perp}$, and hence%
  \footnote{%
    This is a general fact: if $X$ is positive and $Y$ is selfadjoint with $Y \ge -X$, then $\trace(Y_-) \le \trace(X)$.
    Indeed, if $P$ is the range projection of $Y_-$, then we have $Y_- = P(-Y)P \le PXP$, hence $\trace(Y_-) \le \trace(PXP) \le \trace(X)$.%
  }
  \begin{equation}
    \label{eq:tr-R-perp-Y_k-R-perp-neg-small}
    \trace(R^{\perp} Y_k R^{\perp})_- \le \trace(R^{\perp} X R^{\perp}) \le \trace(Q^{\perp} X Q^{\perp}) < \delta = \frac{\epsilon}{3(\norm{A} + 1)}.
  \end{equation}
  Since $Y_k \xrightarrow{\wstar} 0$ and $R \in \K$, there is some $N_2$ such that for all $k \ge N_2$,
  \begin{equation}
    \label{eq:tr-R-Y_k-small}
    \abs{\trace(RY_k)} = \abs{\trace(RY_k R)} < \eta := \min \set*{ \frac{\epsilon}{3(\norm{A} + 1)}, \frac{c}{2} }.
  \end{equation}
  By \Cref{lem:weak-star-tail-convergence} there is some $N_3$ such that for all $k \ge N_3$,
  \begin{equation}
    \label{eq:tr-R-perp-Y_k-R-perp-A-close-to-tr-Y_k-A}
    \abs{\trace(R^{\perp} Y_k R^{\perp} A) - \trace(Y_k A)} < \zeta := \frac{\epsilon}{3}.
  \end{equation}

  Now, any selfadjoint trace-class operator $X$ is majorized by the rank-2 selfadjoint operator $X'$ whose nonzero eigenvalues\footnotemark{} are $\trace(X_+) = \trace(X) + \trace(X_-)$ and $-\trace(X_-)$.

  \footnotetext{%
    if either or both of these eigenvalues are zero, then $X'$ has rank 1 or 0.%
  }

  Consequently, applying this fact to $R^{\perp} Y_N R^{\perp}$, by \Cref{cor:subspace-restrction-rank-2-majorization} there is some $Y \in \orbit((R^{\perp} Y_N R^{\perp})')$ with $R_Y \le R_{R^{\perp} Y_N R^{\perp}}$, where $N := \max_{1 \le i \le 3} N_i$, for which $\trace(YA) = \trace(R^{\perp} Y_N R^{\perp}A)$.
  Thus, $RY = YR = 0$.
  Let $y,y' \in R^{\perp} \Hil$ be the unit eigenvectors of $Y$ corresponding to the positive and negative eigenvalues.
  In our case, since $Y \in \orbit((R^{\perp} Y_N R^{\perp})')$, and $\lambda(R^{\perp} Y_N R^{\perp}) \maj \lambda((R^{\perp} Y_N R^{\perp})')$, and using \eqref{eq:tr-R-Y_k-small},
  \begin{equation*}
    \trace(Y) = \trace((R^{\perp} Y_N R^{\perp})') = \trace(R^{\perp} Y_N R^{\perp}) = \trace(Y_N) - \trace(RY_N) \ge c - \eta > 0,
  \end{equation*}
  so at least $y$ exists.
  If $Y$ has no negative eigenvalue, simply set $y' = 0$.
  Then
  \begin{equation}
    \label{eq:y-y'-eigenvectors}
    \trace(R^{\perp} Y_N R^{\perp} A) = \trace(YA) = \trace(R^{\perp} Y_N R^{\perp})_+ \innerprod{Ay}{y} + \trace(R^{\perp} Y_N R^{\perp})_- \innerprod{Ay'}{y'}.
  \end{equation}
  Note that
  \begin{align*}
    \trace(R^{\perp} Y_N R^{\perp})_+
    &= \trace(R^{\perp} Y_N R^{\perp}) + \trace(R^{\perp} Y_N R^{\perp})_- \\
    &= \trace(Y_N) - \trace(RY_N) + \trace(R^{\perp} Y_N R^{\perp})_- \\
    &= c - \trace(RY_N) + \trace(R^{\perp} Y_N R^{\perp})_-.
  \end{align*}
  Therefore, by the previous display and rearranging \eqref{eq:y-y'-eigenvectors},
  \begin{equation}
    \label{eq:c-innerprod-Ay-y}
    \begin{aligned}
      c \innerprod{Ay}{y} &= (\trace(RY_N) - \trace(R^{\perp} Y_N R^{\perp})_- + \trace(R^{\perp} Y_N R^{\perp})_+) \innerprod{Ay}{y} \\
      &= (\trace(RY_N) - \trace(R^{\perp} Y_N R^{\perp})_-) \innerprod{Ay}{y} \\
      &\qquad - \trace(R^{\perp} Y_N R^{\perp})_- \innerprod{Ay'}{y'} + \trace(R^{\perp} Y_N R^{\perp}A).
    \end{aligned}
  \end{equation}

  Finally, we obtain by applying \eqref{eq:c-innerprod-Ay-y}, the triangle inequality and the inequalities, \labelcref{eq:tr-Y_k-A-close-to-mu,eq:tr-R-perp-Y_k-R-perp-neg-small,eq:tr-R-Y_k-small,eq:tr-R-perp-Y_k-R-perp-A-close-to-tr-Y_k-A}
  \begin{align*}
    \abs{c \innerprod{Ay}{y} - \mu} &\le \abs*{ (\trace(RY_N) - \trace(R^{\perp} Y_N R^{\perp})_-) \innerprod{Ay}{y} - \trace(R^{\perp} Y_N R^{\perp})_- \innerprod{Ay'}{y'} } \\
                                    &\qquad  + \abs{ \trace(R^{\perp} Y_N R^{\perp}A) - \mu } \\
                                    &\le (\eta + \delta) \norm{A} + \delta \norm{A} + \abs{ \trace(R^{\perp} Y_N R^{\perp}A) - \trace(Y_N A) } + \abs{ \trace(Y_N A) - \mu } \\
                                    &\le (\eta + 2\delta)\norm{A} + \zeta + \gamma < \epsilon.
  \end{align*}
  Since $y \in R^{\perp} \Hil \subseteq P^{\perp} \Hil$, this completes the proof.
\end{proof}

\Cref{lem:key-lemma} leads to a dichotomy for \weakstar{} convergent sequences $\seq{X_k}$ in $\orbit(C)$ for which $\trace(X_k A)$ converges in $\closure{\ocnr(A)}$.

\begin{proposition}
  \label{prop:weak-star-convergence-essential-numerical-range}
  Let $C \in \traceclass^+$ be a positive trace-class operator and consider a sequence $\seq{X_k}$ in $\orbit(C)$ converging to $X$ in the \weakstar{} topology on $\traceclass$.
  If $A \in B(\Hil)$ and $\trace(X_k A) \to x$, then either
  \begin{enumerate}
  \item $\trace(X) = \trace(C)$, in which case $X_k \to X$ in trace norm; or
  \item $\trace(X) < \trace(C)$, in which case $x - \trace(XA) \in \trace(C-X) \essnr(A)$.
  \end{enumerate}
\end{proposition}

\begin{proof}
  Since the set $\setb{ Z \in \traceclass^+ }{ \lambda(Z) \submaj \lambda(C) }$, is \weakstar{} closed by \Cref{cor:submajorization-weak-star-compact} and it contains $\orbit(C)$, we see that $\lambda(X) \submaj \lambda(C)$, and hence $\trace(X) \le \trace(C)$.

  If $\trace(X) = \trace(C)$, then since $X,X_k,C$ are all positive trace-class operators,
  \begin{equation*}
    \norm{X_k}_1 = \trace(X_k) = \trace(C) = \trace(X) = \norm{X}_1.
  \end{equation*}
  Therefore $X_k \xrightarrow{\wstar} X$ and $\norm{X_k}_1 \to \norm{X}_1$, and consequently $X_k \xrightarrow{\norm{\bigcdot}_1} X$, which is due to Arazy and Simon \cite{Ara-1981-PAMS,Sim-1981-PAMS}.

  If $\trace(X) < \trace(C)$, then set $Y_k := X_k - X$ and notice
  \begin{enumerate}
  \item $\trace(Y_k) = \trace(X_k) - \trace(X) = \trace(C - X)$ is a positive constant;
  \item $Y_k \xrightarrow{\wstar} 0$;
  \item $Y_k = X_k - X \ge -X$.
  \end{enumerate}
  Moreover, $\trace(Y_k A) = \trace(X_k A) - \trace(XA) \to x - \trace(XA)$.
  Thus $\seq{Y_k}_{k=1}^{\infty}$ satisfies the hypotheses of \Cref{lem:key-lemma} with $\mu = x - \trace(XA)$.

  Then we inductively construct an orthonormal sequence $\seq{y_n}_{n=1}^{\infty}$ for which $\abs{\trace(C-X) \innerprod{Ay_n}{y_n} - (x - \trace(XA))} < \frac{1}{n}$.
  We do this as follows: apply \Cref{lem:key-lemma} to obtain $y_1 \in \Hil$ for which $\abs{\trace(C-X) \innerprod{Ay_1}{y_1} - (x - \trace(XA))} < 1$.
  Then suppose for $m \in \nats$ we have constructed $y_1,\ldots,y_m$.
  Then again apply \Cref{lem:key-lemma} with the projection onto $\spans \set{y_1,\ldots,y_m}$ to obtain $y_{m+1} \in \set{y_1,\ldots,y_m}^{\perp}$ for which $\abs{\trace(C-X) \innerprod{Ay_{m+1}}{y_{m+1}} - (x - \trace(XA))} < \frac{1}{m+1}$.

  Having constructed the desired orthonormal sequence $\seq{y_n}_{n=1}^{\infty}$, we note that since $\trace(C-X) > 0$, $\innerprod{Ay_n}{y_n} \to \frac{x - \trace(XA)}{\trace(C-X)}$, and therefore the limit lies in $\essnr(A)$.
  Hence $x - \trace(XA) \in \trace(C-X) \essnr(A)$.
\end{proof}

We are now in position to establish our main theorem.

\begin{theorem}
  \label{thm:main-theorem}
  Let $C \in \traceclass^+$ be a positive trace-class operator and let $A \in B(\Hil)$.
  Then
  \begin{align*}
    \closure{\ocnr(A)} &= \setb1{ \trace(XA) + \trace(C-X) \essnr(A) }{ X \in \traceclass^+, \lambda(X) \submaj \lambda(C) } \\
                       &= \conv \bigcup_{0 \le m \le \rank(C)} \big( \ocnr[C_m](A) + \trace(C-C_m) \essnr(A) \big),
  \end{align*}
  where $C_m = \diag(\lambda_1(C),\ldots,\lambda_m(C),0,0,\ldots)$.
\end{theorem}

\begin{proof}
  We will prove the set equalities by establishing three subset inclusions.

  We begin by proving
  \begin{equation}
    \label{eq:closure-ocnr-subseteq-submajorization}
    \closure{\ocnr(A)} \subseteq \setb1{ \trace(XA) + \trace(C-X) \essnr(A) }{ X \in \traceclass^+, \lambda(X) \submaj \lambda(C) }.
  \end{equation}
  Take any $x \in \closure{\ocnr(A)}$.
  Then there is a sequence $X_k \in \orbit(C)$ for which $\trace(X_k A) \to x$.
  Since the set $\setb{ Z \in \traceclass^+ }{ \lambda(Z) \submaj \lambda(C) }$ contains $\orbit(C)$ and is \weakstar{} compact and metrizable by \Cref{cor:submajorization-weak-star-compact}, it is \weakstar{} sequentially compact.
  Therefore, by passing to a subsequence we may assume $X_k \xrightarrow{\wstar} X \in \traceclass^+$ with $\lambda(X) \submaj \lambda(C)$.
  By \Cref{prop:weak-star-convergence-essential-numerical-range}, either $X_k \to X$ in trace norm or $x - \trace(XA) \in \trace(C-X) \essnr(A)$.
  In case of the latter, there is nothing more to prove, since $x \in \trace(XA) + \trace(C-X) \essnr(A)$.
  In case $X_k \to X$ in trace norm, then
  \begin{equation*}
    \abs{\trace((X_k - X)A)} \le \norm{X_k - X}_1 \norm{A} \to 0,
  \end{equation*}
  hence $\trace(X_k A) \to \trace(XA)$, and therefore as $\trace(X) = \trace(C)$,
  \begin{equation*}
    x = \trace(XA) = \trace(XA) + 0 \cdot \essnr(A) = \trace(XA) + \trace(C-X) \essnr(A).
  \end{equation*}

  Next we prove
  \begin{equation}
    \label{eq:submajorization-subseteq-convex-hull-union}
    \begin{gathered}
      \setb1{ \trace(XA) + \trace(C-X) \essnr(A) }{ X \in \traceclass^+, \lambda(X) \submaj \lambda(C) } \\
      \subseteq \conv \bigcup_{0 \le m \le \rank(C)} \big( \ocnr[C_m](A) + \trace(C-C_m) \essnr(A) \big).
    \end{gathered}
  \end{equation}
  This follows easily from \Cref{lem:submajorization-convex-hull-of-orbits}.
  In particular, consider $X \in \traceclass^+, \lambda(X) \submaj \lambda(C)$.
  Then by \Cref{lem:submajorization-convex-hull-of-orbits} there is some $Y \in \traceclass^+$ for which $\lambda(X) \maj \lambda(Y) \submaj \lambda(C)$, and
  $\orbit(Y) \subseteq \conv \bigcup \setb{ \orbit(C_m) }{ 0 \le m \le \rank(C) }$.
  By \Cref{thm:c-numerical-range-via-majorization}(ii), there is some $Z \in \orbit(Y)$ for which $\trace(ZA) = \trace(XA)$, and moreover, $\trace(Z) = \trace(Y) = \trace(X)$.
  Thus
  \begin{align*}
    \trace(XA) + \trace(C-X) \essnr(A)
    &= \trace(ZA) + \trace(C-Z) \essnr(A) \\
    &\subseteq \conv \bigcup_{0 \le m \le \rank(C)} \big( \ocnr[C_m](A) + \trace(C-C_m) \essnr(A) \big),
  \end{align*}
  establishing \eqref{eq:submajorization-subseteq-convex-hull-union}.

  Next we will show
  \begin{equation}
    \label{eq:convex-hull-union-subseteq-closure-ocnr}
    \conv \bigcup_{0 \le m \le \rank(C)} \big( \ocnr[C_m](A) + \trace(C-C_m) \essnr(A) \big) \subseteq \closure{\ocnr(A)}.
  \end{equation}
  Since $\ocnr(A)$ is convex by \Cref{thm:c-numerical-range-via-majorization}(i), and because the closure of a convex set is convex, it suffices to prove for all $0 \le m \le \rank(C)$,
  \begin{equation*}
    \ocnr[C_m](A) + \trace(C-C_m) \essnr(A) \subseteq \closure{\ocnr(A)}.
  \end{equation*}
  Now, if $m = \rank(C)$, then $\orbit(C_m) = \orbit(C)$, so $\ocnr[C_m](A) = \ocnr(A)$ and $\trace(C-C_m) = 0$, so there is nothing to prove.

  So suppose $m < \rank(C)$, which implies that $C_m$ is finite rank.
  Then take any $X \in \orbit(C_m)$ and $\mu \in \essnr(A)$, and let $\epsilon > 0$.
  Letting $R_X$ denote the range projection of $X$, which is finite, we see that the compression $R_X^{\perp} A \vert_{R_X^{\perp} \Hil}$ of $A$ to $R_X^{\perp} \Hil$ satisfies $\mu \in \essnr(A) = \essnr(R_X^{\perp} A \vert_{R_X^{\perp} \Hil})$.
  Therefore, there exists an orthonormal sequence of vectors $\seq{y_n}_{n=1}^{\infty}$ in $R_X^{\perp} \Hil$ for which $\abs{\innerprod{Ay_n}{y_n} - \mu} < \frac{\epsilon}{\trace(C-C_m)}$.
  Let $X'$ be the diagonal operator (relative to an orthonormal basis containing $\seq{y_n}_{n=1}^{\infty}$) defined by $X'y_n = \lambda_{m+n}(C) y_n$, and which is zero on the orthogonal complement of $\spans \setb{y_n}{n \in \nats}$.
  Since $X \in \orbit(C_m)$, $X' \in \orbit(\diag(\lambda_{m+1}(C),\lambda_{m+2}(C),\ldots))$, and $XX' = X'X = 0$, then $X+X' \in \orbit(C)$, $\trace(X') = \trace(C-X) = \trace(C-C_m)$ and
  \begin{align*}
    \abs{\trace((X+X')A) - (\trace(XA) + \trace(C-C_m)\mu)}
    &= \abs{\trace(X'A) - \trace(C-C_m) \mu} \\
    &\le \trace(C-C_m) \frac{\epsilon}{\trace(C-C_m)} = \epsilon.
  \end{align*}
  Because $\epsilon$ was arbitrary and $\trace((X+X')A) \in \ocnr(A)$, this proves $\trace(XA) + \trace(C-C_m) \mu \in \closure{\ocnr(A)}$.
  Since $X \in \orbit(C_m)$ and $\mu \in \essnr(A)$ were arbitrary, we have verified \eqref{eq:submajorization-subseteq-convex-hull-union}.

  Finally, combining the subset relations \labelcref{eq:closure-ocnr-subseteq-submajorization,,eq:submajorization-subseteq-convex-hull-union,,eq:convex-hull-union-subseteq-closure-ocnr} proves the theorem.
\end{proof}

\begin{example}
  \label{ex:selfadjoint-fail}
  We note that the most na\"ive extension of \Cref{thm:main-theorem} to selfadjoint $C$ is \emph{false}.
  That is, one might wonder if:
  \begin{equation*}
    \closure{\ocnr(A)} \overset{?}{=} \conv \bigcup_{0 \le m_{\pm} \le \rank(C_{\pm})} \big( \ocnr[C_{m_-,m_+}](A) + \trace(C-C_{m_-,m_+}) \essnr(A) \big),
  \end{equation*}
  where $C_{m_-,m_+}$ are defined as in \Cref{lem:submajorization-convex-hull-of-orbits}.

  However, consider $C = \diag(1,1,-1,-1,0,\ldots)$ and $B = \diag(1,1-\frac{1}{2}, 1-\frac{1}{3},\ldots)$ and $A = B \oplus -B$.
  Then
  \begin{align*}
    \ocnr(A) + \trace(C-C_{2,2}) \essnr(A) &= (-4,4) + 0 \cdot [-1,1], \\
    \ocnr[C_{2,1}](A) + \trace(C-C_{2,1}) \essnr(A) &= (-3,3) + 1 \cdot [-1,1], \\
    \ocnr[C_{1,1}](A) + \trace(C-C_{1,1}) \essnr(A) &= [-2,2] + 0 \cdot [-1,1],  \\
    \ocnr[C_{1,0}](A) + \trace(C-C_{1,0}) \essnr(A) &= [-1,1] + 1 \cdot [-1,1],  \\
    \ocnr[C_{0,0}](A) + \trace(C-C_{0,0}) \essnr(A) &= \set{0} + 0 \cdot [-1,1],
  \end{align*}
  and by symmetry considerations we can ignore the others.
  Consequently, the right-hand side of the previous display is the union of these sets, which is $(-4,4)$ and hence not closed.
\end{example}

We conclude this section with a note concerning the $C$-numerical range introduced by Dirr and vom Ende in \cite{DvE-2020-LaMA} (distinct from, but related to, the orbit-closed $C$-numerical range).
The $C$-numerical range, for $C \in \traceclass$, is defined as
\begin{equation*}
  \cnr(A) := \setb{ \trace(XA) }{ X \in \ugroup(C) },
\end{equation*}
and so the difference between $\cnr(A)$ and $\ocnr(A)$ is that the latter allows $X \in \orbit(C) := \closure[\norm{\cdot}_1]{\ugroup(C)}$.
We have neglected mentioning this $C$-numerical range of Dirr and von Ende primarily because, for reasons discussed in \cite{LP-2021-LaMA}, we feel that $\ocnr(A)$ is actually the more natural extension to $C$ infinite rank of the (previously existing, even as early as 1975 in \cite{Wes-1975-LMA}) definition for $C$ finite rank.
The next example reinforces this sentiment by establishing that the second equality in \Cref{thm:main-theorem} \emph{does not hold} if one replaces $\ocnr(A)$ with $\cnr(A)$ everywhere, despite the fact that (by \cite[Theorem~3.1]{LP-2021-LaMA}) $\closure{\cnr(A)} = \closure{\ocnr(A)}$.

\begin{example}
  \label{ex:cnr-fail}
  Let $C \in \traceclass^+$ be a \emph{strictly} positive (i.e., $\ker(C) = \set{0}$) trace-class operator, and let $A \in \K^{sa}$ be a compact selfadjoint operator such that $\rank(A_{\pm}) = \infty$.
  Then the second equality of \Cref{thm:main-theorem} fails for $\cnr(A)$;
  that is,
  \begin{equation}
    \label{eq:cnr-fail}
    \closure{\cnr(A)} \supsetneq \cnr(A) = \conv \bigcup_{0 \le m \le \rank(C)} \big( \cnr[C_m](A) + \trace(C-C_m) \essnr(A) \big).
  \end{equation}

  To understand why, notice that since $A \in \K$, $\essnr(A) = \set{0}$, and so the right-hand side reduces to the convex hull of the union of $\cnr[C_m](A)$ for $0 \le m \le \rank(C) = \infty$.
  We will prove that $\cnr[C_m](A) \subseteq \cnr(A)$ for all $m$, and that $\cnr(A)$ is an open line segment.

  For $m < \infty$, \Cref{prop:c-numerical-range-maximum} and \Cref{thm:c-numerical-range-selfadjoint-formula} guarantee (using the fact that $\ugroup(C_m) = \orbit(C_m)$ since $\rank(C_m) < \infty$; see \cite[Proposition~3.1]{LP-2021-LaMA}) that
  \begin{equation}
    \label{eq:cnr-Cm-A}
    \cnr[C_m](A) = \ocnr[C_m](A) = \left[ -\sum_{n=1}^m \lambda_n(C) \lambda_n^-(A), \sum_{n=1}^m \lambda_n(C) \lambda_n^+(A) \right].
  \end{equation}

  Since $\ugroup(C_{\infty}) = \ugroup(C)$, $\cnr[C_{\infty}](A) = \cnr(A)$ is convex by \cite[Corollary~7.1]{LP-2021-LaMA}.
  Moreover, $\cnr(A) \subseteq \reals$ by \cite[Proposition~3.2]{LP-2021-LaMA}, hence it is an interval by convexity.
  Then \Cref{prop:c-numerical-range-maximum} and \Cref{thm:c-numerical-range-selfadjoint-formula} show
  \begin{equation}
    \label{eq:cnr-C-A}
    \cnr[C_{\infty}](A) = \cnr(A) = \left( -\sum_{n=1}^{\infty} \lambda_n(C) \lambda_n^-(A), \sum_{n=1}^{\infty} \lambda_n(C) \lambda_n^+(A) \right).
  \end{equation}
  In the above, that this interval is open arises from the fact that (due to \Cref{thm:c-numerical-range-selfadjoint-formula}), if $\sup \ocnr(A)$ ($= \sup \cnr(A)$) is attained by some $X \in \orbit(C)$, then $PX = XP = X$ where $P = \chi_{[0,\infty)}(A)$.
  Consequently, since $P \not= I$, $X \notin \ugroup(C)$ and therefore $\sup \cnr(A)$ is not attained.
  A symmetric argument holds for $\inf \cnr(A)$ ($= - \sup \cnr(-A)$).

  Since $\rank(A_{\pm}) = \infty$ and $\rank(C) = \infty$, \eqref{eq:cnr-Cm-A} and \eqref{eq:cnr-C-A} show that $\cnr[C_m](A) \subseteq \cnr(A)$ for every $m$, and that $\cnr(A) \subseteq \reals$ is an open interval, thereby proving \eqref{eq:cnr-fail}.
\end{example}

\section{Inherited closedness}

Our main result in this section is \Cref{thm:closedness-restricts}, which guarantees that if $C \in \traceclass^+$ and $\ocnr(A)$ is closed, then so is $\ocnr[C_m](A)$ for every $0 \le m < \rank(C)$, where $C_m := \diag(\lambda_1(C),\ldots,\lambda_m(C),0,\ldots)$.

The analysis in this section is markedly different from that in the previous section.
Whereas in \Cref{sec:submajorization-closure} we made extensive use of the \weakstar{} topology, in this section such arguments are mostly confined to \Cref{prop:rank-condition-closure}.
Instead, we will make frequent use of a standard technique in the theory of numerical ranges, which essentially allows us to reduce to the case when $A$ is selfadjoint (at least when $C$ is also selfadjoint).

This reduction is due to the following two facts, for $C \in \traceclass^{sa}$ and $a,b \in \complex$, proofs of which are simple, but can be found in \cite[Proposition~3.2]{LP-2021-LaMA}:
\begin{equation*}
  \ocnr(aI + bA) = a \trace(C) + b \ocnr(A) \quad\text{and}\quad \Re(\ocnr(A)) = \ocnr(\Re(A)).
\end{equation*}
Using the first fact, one can reduce the study of points on the boundary of $\ocnr(A)$ to those having maximal real part, because one can simply multiply $A$ by a modulus $1$ constant to rotate the orbit-closed $C$-numerical range.
The second fact allows for the study of points with maximal real part by studying those points which maximize $\ocnr(\Re(A))$.
For points on the boundary of $\ocnr(A)$ which do not lie on a line segment, this reduction often tells the whole story.
But line segments on the boundary are only partially described by this reduction to the selfadjoint case, and often this results in significantly more technical approaches devoted to their study.

We begin this section with a result which, roughly approximated, says: when a significant enough portion of $A$ lies outside the essential spectrum, then the analysis even of entire line segments on the boundary reduces to the selfadjoint case.

\begin{proposition}
  \label{prop:rank-condition-closure}
  Let $C \in \traceclass^+$ be a positive trace-class operator and $A \in B(\Hil)$, and let $m := \max \essspec(\Re(A))$ and let $M := \sup \Re(\ocnr(A))$.
  If $\rank(\Re(A) - mI)_+ \ge \rank(C)$ (i.e., $\trace(\chi_{(m,\infty)}(\Re(A))) \ge \rank(C)$), then
  \begin{equation*}
    \closure{\ocnr(A)} \cap (M + i\reals) = \ocnr(A) \cap (M + i\reals).
  \end{equation*}
\end{proposition}

\begin{proof}
  Clearly, it suffices to assume $m = 0$ by translating the operator $A \mapsto (A - mI)$.

  Take $X_n \in \orbit(C)$ with $\trace(X_n A) \to x \in \closure{\ocnr(A)}$ and $\Re(x) = \sup \Re(\ocnr(A)) = \sup \ocnr(\Re(A))$.
  By \Cref{prop:c-numerical-range-maximum} and \Cref{thm:c-numerical-range-selfadjoint-formula}, we conclude
  \begin{equation}
    \label{eq:real-x-sum-product}
    \Re(x) = \sum_{n=1}^{\infty} \lambda_n(C) \lambda_n(\Re(A)_+).
  \end{equation}

  We claim that $\trace(X_n(\Re(A)_+)) \to \Re(x)$.
  For this, simply notice that $\trace(X_n(\Re(A))) = \Re(\trace(X_n A)) \to \Re(x)$, and also
  \begin{align*}
    \trace(X_n (\Re(A))) &= \trace(X_n (\Re(A)_+)) - \trace(X_n (\Re(A)_-)) \\
                         &\le \trace(X_n (\Re(A)_+)) \\
                         &\le \sum_{n=1}^{\infty} \lambda_n(C) \lambda_n(\Re(A)_+) = \Re(x),
  \end{align*}
  where the last inequality is due to \Cref{prop:c-numerical-range-maximum}.
  Then apply the squeeze theorem.
  
  By the \weakstar{} sequential compactness of $\setb{ Z \in \traceclass^+ }{ \lambda(Z) \submaj \lambda(C) }$ from \Cref{cor:submajorization-weak-star-compact}, there is some $X \in \traceclass^+$ with $\lambda(X) \submaj \lambda(C)$ for which $X_n \to X$ in the \weakstar{} topology.
  Since $\Re(A)_+$ is a compact operator, we see that $\trace(X_n (\Re(A)_+) \to \trace(X(\Re(A)_+)$ and hence
  \begin{equation}
    \label{eq:inequality-for-lemma-5.2-iii}
    \sum_{n=1}^{\infty} \lambda_n(C) \lambda_n(\Re(A)_+) = \Re(x) = \trace(X(\Re(A)_+)) \le \sum_{n=1}^{\infty} \lambda_n(X) \lambda_n(\Re(A)_+).
  \end{equation}

  Invoking \cite[Lemma~5.2(i)]{LP-2021-LaMA} with $\delta_n := \lambda_n(C) - \lambda_n(X)$, which has nonnegative partial sums since $\lambda(X) \submaj \lambda(C)$, we obtain $\sum_{n=1}^N \delta_n \lambda_n(\Re(A)_+) \ge 0$ for all $N$, and taking the limit as $N \to \infty$ we find $\sum_{n=1}^{\infty} \delta_n \lambda_n(\Re(A)_+) \ge 0$.
  Rearranging \eqref{eq:inequality-for-lemma-5.2-iii}, and noting that the individual sums are in $\ell_1$ since $X,C \in \traceclass^+$, we find $\sum_{n=1}^{\infty} \delta_n \lambda_n(\Re(A)_+) \ge 0$ and hence $\sum_{n=1}^{\infty} \delta_n \lambda_n(\Re(A)_+) = 0$.
  Applying \cite[Lemma~5.2(iii)]{LP-2021-LaMA}, we obtain
  \begin{equation}
    \label{eq:sum-delta-equals-zero}
    \sum_{n=1}^N \delta_n = 0, \text{ whenever } \lambda_N(\Re(A)_+) > \lambda_{N+1}(\Re(A)_+).
  \end{equation}
  If $\rank(\Re(A)_+) = \infty$, then \eqref{eq:inequality-for-lemma-5.2-iii} holds for infinitely many $N$ and hence $\sum_{n=1}^{\infty} \delta_n = 0$, in which case $\trace(X) = \trace(C)$. 
  Otherwise, by hypothesis, $M := \rank(\Re(A)_+) \ge \rank(C)$ with $M < \infty$.
  Then, $\lambda_M(\Re(A)_+) > 0 = \lambda_{M+1}(\Re(A)_+)$, and hence
  \begin{equation*}
    \trace(C) - \trace(X) \le \trace(C) - \sum_{n=1}^M \lambda_n(X) = \sum_{n=1}^M \delta_n = 0,
  \end{equation*}
  Since we already have the inequality $\trace(C) \ge \trace(X)$, we conclude $\trace(X) = \trace(C)$.

  Finally,
  \begin{equation*}
    \norm{X_n}_1 = \trace(X_n) = \trace(C) = \trace(X) = \norm{X}_1,
  \end{equation*}
  and hence $X_n \to X$ in trace norm (by the aforementioned result of Arazy and Simon \cite{Ara-1981-PAMS,Sim-1981-PAMS}).
  This guarantees $\trace(X_n A) \to \trace(XA)$, and hence $\trace(XA) = x$.
  Moreover, $\lambda(X) \maj \lambda(C)$, which guarantees that $\trace(XA) \in \ocnr(A)$ by \Cref{thm:c-numerical-range-via-majorization}.
\end{proof}

\begin{remark}
  \label{rem:rank-condition-simpler-proof}
  We note that \Cref{prop:rank-condition-closure} is a significant improvement over \cite[Proposition~5.2]{LP-2021-LaMA} for multitudinous reasons.
  The hypotheses of \Cref{prop:rank-condition-closure} are much weaker, the conclusion is stronger (in the notation of \cite{LP-2021-LaMA}, $\ocnr(A)$ contains the entire line segment $[x_-,x_+]$ instead of simply points arbitrarily close to $x_-$), and the proof is simpler and more elegant.
  
  In addition, \Cref{prop:rank-condition-closure} has, as a direct corollary, the statement that if for every $0 \le \theta < 2\pi$, $\rank(\Re(e^{i\theta}A)- m_{\theta}I)_+ \ge \rank(C)$, then $\ocnr(A)$ is closed, where $m_{\theta} := \max \essspec(e^{i\theta}A)$, which is exactly the content of \cite[Theorem~5.3]{LP-2021-LaMA}.
  However, the proof of \cite[Theorem~5.3]{LP-2021-LaMA} given in that paper was incredibly technical, and so the proof of \Cref{prop:rank-condition-closure} above represents a quite substantial simplification.
  Moreover, \Cref{prop:rank-condition-closure} is a stronger statement than \cite[Theorem~5.3]{LP-2021-LaMA} because it even applies to specific portions of the boundary of $\ocnr(A)$.
\end{remark}

The following result is the main technical lemma needed on the way to proving \Cref{thm:closedness-restricts}.
Note that statements \ref{item:tr-leq-rank-non-closed} and \ref{item:tr-geq-rank-closed} are, in essence, logical inverses.

\begin{lemma}
  \label{lem:extreme-points-ocnr}
  Let $C \in \traceclass^+$ be a positive trace-class operator and let $A \in B(\Hil)$ with $\max \essspec(\Re(A)) = 0$, and set $P := \chi_{(0,\infty)}(\Re(A)), P_0 := \chi_{\set{0}}(\Re(A))$ and $M := \sup \Re(\ocnr(A))$ and $r := \trace(P) = \rank(\Re(A)_+)$.
  \begin{enumerate}
  \item \label{item:boundary-splitting} If $X \in \orbit(C)$ with $\Re(\trace(XA)) = M$, then $X = X_r + X_r' := PXP + P_0 X P_0$ with $PXP \in \orbit(C_r)$.
    
    Moreover, for any $Y \in \traceclass^+$ for which $Y = P_0 Y P_0$,
    \begin{equation*}
      -i \trace(YA) = \trace(YP_0 \Im(A) P_0)) \in \ocnr[Y](A_0)
    \end{equation*}
    where is the compression of $\Im(A)$ to $P_0$.

    In particular, $-i \trace(X_r' A) = \trace(X_r' P_0 \Im(A) P_0)) \in \ocnr[C'](A_0)$, where $C' := \diag(\lambda_{r+1}(C),\lambda_{r+2}(C),\ldots)$.
  \item \label{item:tr-leq-rank-non-closed} If, in addition, there is some $i \nu \in \essnr(A)$ for which $\trace(\chi_{[\nu,\infty)}(A_0)) < \rank(C')$, then there is some $y \in (M+i\reals) \cap \closure{\ocnr(A)}$ such that $\Im(y) > \Im(x)$ for all $x \in (M+i\reals) \cap \ocnr(A)$.
  \item \label{item:tr-geq-rank-closed} Inversely, if for every $i\nu \in \essnr(A)$, $\trace(\chi_{[\nu,\infty)}(A_0)) \ge \rank(C')$, then if $x \in (M+i\reals) \cap \closure{\ocnr(A)}$ has maximal imaginary part among $(M+i\reals) \cap \closure{\ocnr(A)}$, then $x \in \ocnr(A)$.
  \end{enumerate}
\end{lemma}

\begin{proof}
  \hfill \\
  \begin{enumerate}
  \item This follows immediately from \cite[Proposition~5.3]{LP-2021-LaMA} and one small computation.
    Note that in the case when $\rank(\Re(A)_+) = \infty$, then $P_0 X P_0 = 0$.
    By the definition of $P_0$, we find $P_0 A P_0 = P_0 \Re(A) P_0 + i P_0 \Im(A) P_0 = i P_0 \Im(A) P_0$.
    \begin{equation*}
      -i \trace(YA) = -i \trace(P_0 Y P_0 A) = -i \trace(Y P_0 A P_0) = \trace(Y P_0 \Im(A) P_0) = \trace(Y_0 A_0),
    \end{equation*}
    where $Y_0$ is the compression of $Y$ to $P_0$.
    Finally, we apply this $X'_r = P_0 X P_0$, and note that $\orbit(X'_r) = \orbit(C)$ since $\lambda(X'_r) = \lambda(C)$.
  \item Let $M := \sup \Re(\ocnr(A))$ and consider $\ocnr(A) \cap (M + i\reals)$.
    By the convexity of $\ocnr(A)$, this set is either empty, in which case there is nothing to prove (since $\closure{\ocnr(A)} \cap (M + i\reals)$ is nonempty), or else a (not necessarily closed, but possibly degenerate) line segment.
    If this line segment does not contain its upper endpoint, then we are done because this upper endpoint necessarily lies in $\closure{\ocnr(A)} \cap (M + i\reals)$.
    
    So suppose $\ocnr(A) \cap (M + i\reals)$ is a line segment which contains its upper endpoint, and let $x \in \ocnr(A) \cap (M + i\reals)$ denote this element with maximal imaginary part.
    Then there is some $X \in \orbit(C)$ with $\trace(XA) = x$.
    By \ref{item:boundary-splitting} we can decompose $X = X_r + X_r' := PXP + P_0 X P_0$ with $X_r \in \orbit(C_r)$.

    Now, we claim that $-i \trace(X_r' A) = \sup \ocnr[C'](A_0)$.
    If not, there would be some $Z' \in \orbit(C')$ acting on $P_0 \Hil$ for which $\trace(Z'A_0) > -i \trace(X_r'A)$.
    Then setting $Z := PXP + (Z' \oplus \zop_{P_0^{\perp}\Hil}) \in \orbit(C)$, would yield $\trace(ZA) \in \ocnr(A) \cap (M + i\reals)$ with imaginary part exceeding that of $x$, which is a contradiction.

    Let $m = \max \essspec(A_0)$.
    Then by \Cref{prop:c-numerical-range-maximum} and \Cref{thm:c-numerical-range-selfadjoint-formula},
    \begin{equation*}
      -i \trace(X_r' A) = m \trace(C') + \sum_{n=1}^{\rank(C')} \lambda_n(A_0 - mI)_+ \lambda_n(C').
    \end{equation*}

    Suppose there is some $i \nu \in \essnr(A)$ for which $k := \trace(\chi_{[\nu,\infty)}(A_0)) < \rank(C')$.
    Since $\sup \ocnr[C'](A_0)$ is attained (by the compression of $X_r'$ to $P_0\Hil$), \Cref{thm:c-numerical-range-selfadjoint-formula} guarantees that $\rank(C') \le \trace(\chi_{[m,\infty)}(A_0))$ and therefore $m < \nu$.
    Then let $\set{e_n}_{n=1}^k$ be the eigenvectors corresponding to the $k$ largest eigenvalues of $A_0$ (i.e., $\set{m + \lambda_n(A_0 - mI)_+}_{n=1}^k$).
    The definition of $k$ guarantees that $m + \lambda_n(A_0 - mI)_+ < \nu$ for all $n > k$.
    Consequently,
    \begin{equation*}
      \sum_{n=k+1}^{\rank(C')} (m + \lambda_n(A_0 - mI)_+) \lambda_n(C') < \sum_{n=k+1}^{\rank(C')} \nu \lambda_n(C') = \nu \trace(C - C_{r+k}).
    \end{equation*}
    Let $C'' := \diag(\lambda_{r+1}(C),\ldots,\lambda_{r+k}(C),0,\ldots)$, and select $X'' \in \orbit(C'')$ such that $e_n$ is the eigenvector of $\lambda_{r+n}(C)$ for each $1 \le n \le k$.
    Then
    \begin{align*}
      -i \trace(X_r' A)
      &= m\trace(C') + \sum_{n=1}^{\rank(C')} \lambda_n(A_0 - mI)_+ \lambda_n(C') \\
      &= \sum_{n=1}^{\rank(C')} (m + \lambda_n(A_0 - mI)_+) \lambda_n(C') \\
      &< \sum_{n=1}^k (m + \lambda_n(A_0 - mI)_+) \lambda_n(C') + \nu \trace(C-C_{r+k}) \\
      &= -i \trace(X''A) + \nu \trace(C-C_{r+k}).
    \end{align*}
    Now $X_k := X_r + X'' \in \orbit(C_{r+k})$, and
    $\Re(\trace(X_k A) + \trace(C-C_{r+k})i\nu ) = \Re(\trace(X_r A)) = M$ and $y := \trace(X_k A) + \trace(C-C_{r+k}) i \nu \in \closure{\ocnr(A)}$ by
    \Cref{thm:main-theorem}, so $y \in (M+i\reals) \cap \closure{\ocnr(A)}$.
    Moreover,
    \begin{align*}
      \Im(\trace(X_k A) + \trace(C-C_{r+k})i \nu)
      &= \Im(\trace(X_r A)) -i \trace(X'' A) + \trace(C-C_{r+k}) \nu  \\
      &> \Im(\trace(X_r A)) -i \trace(X_r' A) = \Im(\trace(XA)),
    \end{align*}
    so $\Im(y) > \Im(x)$.
  \item Since $\max \essspec(\Re(A)) = 0$, notice that $\max \Re(\essnr(A)) = \max \essnr(\Re(A)) = 0$ also.
    Since $x \in \closure{\ocnr(A)}$ is an extreme point, \Cref{thm:main-theorem} guarantees that there is some $k \le \rank(C)$, $X_k \in \orbit(C_k)$ and $\mu + i \nu \in \essnr(A)$ such that
    \begin{equation*}
      x = \trace(X_k A) + \trace(C-C_k) (\mu + i \nu).
    \end{equation*}
    Obviously, if $k = \rank(C)$, then $C_k = C$ and hence $x = \trace(X_k A) \in \ocnr(A)$.

    So suppose that $k < \rank(C)$, and hence also $\trace(C-C_k) > 0$.
    Note that since $\max \Re(\essnr(A)) = 0$, $\mu \le 0$ and applying \Cref{prop:c-numerical-range-maximum} and \Cref{thm:c-numerical-range-selfadjoint-formula},
    \begin{align*}
      \Re(x) &= \Re(\trace(X_k A)) + \trace(C-C_k) \mu \le \Re(\trace(X_k A)) \\
             &\le \sup \Re(\ocnr[C_k](A)) = \sum_{n=1}^k \lambda_n(C) \lambda^+_n(\Re(A)) \\
             &\le \sum_{n=1}^{\rank(C)} \lambda_n(C) \lambda^+_n(\Re(A)) = \sup \Re(\ocnr(A)),
    \end{align*}
    but the first and last expressions are equal, so we must have equality throughout.
    Thus $\mu = 0$ and $\Re(\trace(X_k A)) = \sup \Re(\ocnr[C_k](A)) = M$.
    By \ref{item:boundary-splitting} we can decompose $X_k = X_r + X_r' := PX_kP + P_0 X_k P_0$ with $X_r \in \orbit(C_r)$ and $X_r' \in \orbit(C'')$, where $C'' := \diag(\lambda_{r+1}(C),\ldots,\lambda_k(C),0,\ldots)$.
    Now, setting $m := \max \essspec(A_0)$ (or in case $P_0$ is finite so that $A_0$ acts on a finite dimensional space, select $m := \min \spec(A_0)$, which guarantees $m + \lambda_n(A_0 - m I)_+ = \lambda_n(A_0)$ for all $n$), then by \ref{item:boundary-splitting},
    \begin{equation*}
      -i \trace(X_r' A) \le \sup \ocnr[C''](A_0) = \sum_{n=1}^{k-r} \lambda_{r+n}(C) (m + \lambda_n(A_0 - mI)_+)
    \end{equation*}
    By hypothesis, we know that $\nu \le m + \lambda_n(A_0 - mI)_+$ for all $1 \le n \le \rank(C')$.
    Therefore,
    \begin{equation*}
       \trace(C-C_k) \nu = \sum_{n=k-r+1}^{\rank(C')} \lambda_{r+n}(C) \nu \le \sum_{n=k-r+1}^{\rank(C')} \lambda_{r+n}(C) (m + \lambda_n(A_0 - mI)_+).
    \end{equation*}

    Either $P_0$ is finite, so that $A_0$ acts on a finite dimensional space, in which case $\sup \ocnr[C'](A_0)$ is attained by compactness of the unitary group in finite dimensions, or else $P_0$ is infinite.
    In the latter case, since $m \in \essspec(A_0) \subseteq \essnr(A_0)$, there is some orthonormal sequence of vectors $x_n \in P_0 \Hil$ for which $\innerprod{A_0 x_n}{x_n} \to m$, in which case $\innerprod{A x_n}{x_n} \to im$, and hence $im \in \essnr(A)$.
    So by hypothesis, $\trace(\chi_{[m,\infty)}(A_0)) \ge \rank(C')$.
    Therefore, by \Cref{thm:c-numerical-range-selfadjoint-formula}, $\sup \ocnr[C'](A_0)$ is attained.

    Thus, regardless of whether $P_0$ is finite or infinite $\sup \ocnr[C'](A_0)$ is attained by some $\trace(YA_0)$, with $Y \in \orbit(C')$ acting on $P_0 \Hil$.
    Then set $X' := Y \oplus \zop_{P_0^{\perp} \Hil} \in \orbit(C')$ and notice $X := X_r + X' \in \orbit(C)$.
    Now, by the choice of $Y$ and the previous two displays,
    \begin{align*}
      -i \trace(X'A) &= \trace(YA_0) = \sup \ocnr[C'](A_0) \\
                     &= \sum_{n=1}^{\rank(C')} \lambda_{r+n}(C) (m + \lambda_n(A_0 - mI)_+) \\
                     &\ge -i\trace(X_r'A) + \trace(C-C_k) \nu.
    \end{align*}
    Consequently, if we let $y = \trace(XA) \in \ocnr(A)$, then $\Re(y) = M = \Re(x)$.
    Moreover,
    \begin{align*}
      \Im(y) &= \Im(\trace(X_r A)) + \trace(YA_0) \\
             &\ge \Im(\trace(X_r A)) + \trace(C - C_k)\nu - i\trace(X_r' A) \\
             &\ge \Im(\trace(X_r A)) + \trace(C - C_k)\nu + \trace(X_r' \Im(A)) \\
             &= \Im(x).
    \end{align*}
    and by the hypothesis on $x$, we also have $\Im(x) \ge \Im(y)$.
    Therefore, $x = y \in \ocnr(A)$. \qedhere
  \end{enumerate}
\end{proof}

Using \Cref{lem:extreme-points-ocnr}, we can bootstrap it into \Cref{prop:restriction-preserves-closed-boundary} by making use of \Cref{thm:main-theorem} to conclude that if a portion of the boundary is closed (i.e., if the intersection of a supporting line with $\closure{\ocnr(A)}$ is contained within $\ocnr(A)$), then this property is inherited by all $\ocnr[C_m](A)$ with $0 \le m < \rank(C)$ (i.e., the intersection of $\closure{\ocnr[C_m](A)}$ with a supporting line parallel to the one for $\closure{\ocnr(A)}$ is contained within $\ocnr[C_m](A)$).

\begin{proposition}
  \label{prop:restriction-preserves-closed-boundary}
  Let $C \in \traceclass^+$ be a positive trace-class operator and let $A \in B(\Hil)$, and set $M := \sup \Re(\ocnr(A))$.
  If
  \begin{equation*}
    \closure{\ocnr(A)} \cap (M+i\reals) = \ocnr(A) \cap (M+i\reals),
  \end{equation*}
  then for all $0 \le m < \rank(C)$,
  \begin{equation*}
    \closure{\ocnr[C_m](A)} \cap (M+i\reals) = \ocnr[C_m](A) \cap (M+i\reals).
  \end{equation*}
\end{proposition}

\begin{proof}
  By translating, we may assume without loss of generality that $\max \essspec(\Re(A)) = 0$.
  Note that if $\rank(C_m) = m \le \rank(\Re(A)_+) = \trace(\chi_{(0,\infty)}(\Re(A)))$, then the claim follows from \Cref{prop:rank-condition-closure}.
  So we may suppose $\trace(\chi_{(0,\infty)}(\Re(A))) < \rank(C_m)$.
  Additionally, by the hypothesis on $\ocnr(A)$, \Cref{thm:c-numerical-range-selfadjoint-formula} guarantees that $\rank(C) \le \trace(\chi_{[0,\infty)}(\Re(A)))$.

  Set $P := \chi_{(0,\infty)}(\Re(A))$ and $r := \trace(P) = \rank(\Re(A)_+)$, and 
  set $P_0 := \chi_{\set{0}}(\Re(A))$ and $M := \sup \Re(\ocnr(A))$.
  Note that $r = \trace(P) < \rank(C_m) = m$.
  Take $x \in \closure{\ocnr(A)}$ such that $\Re(x) = M$ and for $x$ has maximal imaginary part among $\closure{\ocnr(A)} \cap (M+i\reals)$.
  The hypothesis implies $x \in \ocnr(A)$, and therefore by the contrapositive of \Cref{lem:extreme-points-ocnr}\ref{item:tr-leq-rank-non-closed}, for every $i \nu \in \essnr(A)$, $\trace(\chi_{[\nu,\infty)}(A_0)) \ge \rank(C') \ge \rank(C'')$, where $C'' := \diag(\lambda_{r+1}(C),\ldots,\lambda_m(C),0,\ldots)$.
  Then by \Cref{lem:extreme-points-ocnr}\ref{item:tr-geq-rank-closed}, for $y \in \closure{\ocnr[C_m](A)}$ with $\Re(y) = \sup \Re(\ocnr[C_m](A)) = M$, and since $\trace(\chi_{[\nu,\infty)}(A_0)) \ge \rank(C'')$ and $y$ having maximal imaginary part among $\closure{\ocnr[C_m](A)} \cap (M+i\reals)$, we have $y \in \ocnr[C_m](A)$.

  Applying the above argument to $A^{*}$ proves that for $z \in \closure{\ocnr[C_m](A)}$ with $\Re(z) = \sup \Re(\ocnr[C_m](A)) = M$ and $z$ having minimal imaginary part among $\closure{\ocnr[C_m](A)} \cap (M+i\reals)$, we have $z \in \ocnr[C_m](A)$.
  Since every element of $\closure{\ocnr[C_m](A)} \cap (M+i\reals)$ is a convex combination of $y,z$, and since $\ocnr[C_m](A)$ is convex, we conclude $\closure{\ocnr[C_m](A)} \cap (M+i\reals) = \ocnr[C_m](A) \cap (M+i\reals)$.
\end{proof}

\begin{theorem}
  \label{thm:closedness-restricts}
  Let $C \in \traceclass^+$ be a positive trace-class operator and let $A \in B(\Hil)$.
  If $\ocnr(A)$ is closed, then for all $0 \le m < \rank(C)$, $\ocnr[C_m](A)$ is closed.
\end{theorem}

\begin{proof}
  Apply \Cref{prop:restriction-preserves-closed-boundary} to $e^{i\theta} A$ for each $0 \le \theta < 2\pi$ and note that $\ocnr(e^{i\theta}(A)) = e^{i\theta} \ocnr(A)$.
\end{proof}

\begin{corollary}
  \label{cor:closed-inclusion-chain}
  Let $C \in \traceclass^+$ be a positive trace-class operator and let $A \in B(\Hil)$.
  Then $\ocnr(A)$ is closed if and only if
  \begin{equation}
    \label{eq:inclusion-chain}
    \begin{aligned}
      \trace(C) \essnr(A) &\subseteq \ocnr[C_1](A) + \trace(C-C_1) \essnr(A) \\
      &\subseteq \ocnr[C_2](A) + \trace(C-C_2) \essnr(A) \\
      &\ \,\vdots \\
      &\subseteq \ocnr(A).
    \end{aligned}
  \end{equation}
\end{corollary}

\begin{proof}
  $(\Leftarrow)$
  This follows immediately from the chain of inclusions \eqref{eq:inclusion-chain}, and \Cref{thm:c-numerical-range-via-majorization,thm:main-theorem}.

  $(\Rightarrow)$
  Suppose that $\ocnr(A)$ is closed.
  Then by \Cref{thm:closedness-restricts}, $\ocnr[C_m](A)$ is closed for every $0 \le m < \rank(C)$.
  Fix an arbitrary $0 \le m < m+1 < \rank(C)$.
  Since $\ocnr[C_{m+1}](A)$ is closed, \Cref{thm:main-theorem} applied to $\ocnr[C_{m+1}](A)$ guarantees
  \begin{equation*}
    \ocnr[C_m](A) + \trace(C_{m+1} - C_m) \essnr(A) \subseteq \ocnr[C_{m+1}](A).
  \end{equation*}
  Therefore, adding $\trace(C-C_{m+1}) \essnr(A)$ to both sides,
  \begin{equation*}
    \ocnr[C_m](A) + \trace(C-C_m) \essnr(A) \subseteq \ocnr[C_{m+1}](A) + \trace(C-C_{m+1}) \essnr(A).
  \end{equation*}
  Moreover, \Cref{thm:main-theorem} also guarantees that $\ocnr(A)$ contains the entire chain.
\end{proof}

\bibliographystyle{tfnlm}
\bibliography{references.bib}

\end{document}